\documentclass[final,nomarks]{dmtcs-episciences}

\usepackage[utf8]{inputenc}
\usepackage{subfigure}
\usepackage{tikz,tikz-qtree}
\usepackage{enumerate}

\usepackage{geometry}
\usepackage{adjustbox}
\newtheorem{theorem}{Theorem}
\newtheorem{proposition}{Proposition}

\newtheorem{corollary}{Corollary}

\newtheorem{definition}{Definition}

\newtheorem{remark}{Remark}

\newcommand{\speci}{$k$-special} 
\newcommand{\mod}{ {~\rm mod~}}

\author{Andr\'e K\"undgen\affiliationmark{1}\thanks{Supported by ERC Advanced Grant GRACOL, project no. 320812.}
  \and Tonya Talbot\affiliationmark{1} 
  }
\title{Nonrepetitive edge-colorings of trees} 
\affiliation{
  California State University San Marcos, San Marcos, CA, USA}

\received{2017-1-17}
\accepted{2017-6-4}

\begin{document}
\bibliographystyle{siam}
\publicationdetails{19}{2017}{1}{18}{2651}
\maketitle

\begin{abstract}
 A repetition is a sequence of symbols in which the first half is the
same as the second half. An edge-coloring of a graph is repetition-free or nonrepetitive if
there is no path with a color pattern that is a repetition. The minimum
number of colors so that a graph has a nonrepetitive edge-coloring is called its Thue edge-chromatic number. 

We improve on the best known general upper bound of $4\Delta-4$ for the Thue edge-chromatic number of trees of maximum degree $\Delta$ 
due to Alon, Grytczuk, Haluszczak and Riordan (2002) by providing a simple nonrepetitive edge-coloring with $3\Delta-2$ colors.
\end{abstract}


\section{Introduction}
\label{sec:intro}

A {\em repetition} is a sequence of even length (for example $abacabac$), such that the first half of the sequence is identical to the second half. 
In 1906 Thue~\cite{T1} proved that there are infinite sequences of 3 symbols that do not contain a repetition consisting of consecutive elements in the sequence. Such sequences are called {\em Thue sequences}. Thue studied these sequences as words that do not contain any square words $ww$ and the interested reader can consult Berstel~\cite{B2,B1} for some background and a translation of Thue's work using more current terminology.
Thue sequences have been studied and generalized in many views (see the survey of Grytczuk~\cite{G}), but in this paper we focus on the natural generalization of the Thue problem to Graph Theory.

In 2002 Alon, Grytczuk, Ha{\l}uszczak and Riordan~\cite{AGHR} proposed calling a coloring of the edges of a graph {\em nonrepetitive} if the sequence of colors on any open path in $G$ is nonrepetitive. We will use $\pi'(G)$ to denote the {\em Thue chromatic index} of a graph $G$, which is the minimum number of colors in a nonrepetitive edge-coloring of $G$. In~\cite{AGHR} the notation $\pi(G)$ was used for the Thue chromatic index, but by common practice we will instead use this notation for the {\em Thue chromatic number}, which is the minimum number of colors in a nonrepetitive coloring of the {\em vertices} of $G$. Their paper contains many interesting ideas and questions, the most intriguing of which is if $\pi(G)$ is bounded by a constant when $G$ is planar. The best result in this direction is due to Dujmovi{\'c}, Frati, Joret, and Wood~\cite{DFJW} who show that for planar graphs on $n$ vertices $\pi(G)$ is $O(\log n)$. Conjecture~2 from~\cite{AGHR} was settled by Currie~\cite{C} who showed that for the $n$-cycle $C_n$, $\pi(C_n)=3$ when $n\ge 18$. One of the conjectures from~\cite{AGHR} that remains open is whether $\pi'(G)=O(\Delta)$ when $G$ is a graph of maximum degree $\Delta$. At least $\Delta$ colors are always needed, since nonrepetitive edge-colorings must give adjacent edges different colors.

In this paper we study the seemingly easy question of nonrepetitive edge-colorings of trees. Thue's sequence shows that if $P_n$ is the path on $n$ vertices, then $\pi'(P_n)=\pi(P_{n-1})\le 3$. (Keszegh, Patk{\'o}s, and Zhu~\cite{KPZ} extend this to more general path-like graphs.) Using Thue sequences Alon, Grytczuk, Ha{\l}uszczak and Riordan~\cite{AGHR} proved that every tree of maximum degree $\Delta\geq 2$ has a nonrepetitive edge-coloring with $4(\Delta-1)$ colors and stated that the same method can be used to obtain a nonrepetitive vertex-coloring with 4 colors. However, while the star $K_{1,t}$ is the only tree whose vertices can be colored nonrepetitively with fewer than 3 colors, it is still unknown which trees need 3 colors, and which need 4 (see Bre{\v{s}}ar, Grytczuk, Klav{\v{z}}ar, Niwczyk, Peterin~\cite{BGKNP}.) Interestingly Fiorenzi, Ochem, Ossona de Mendez, and Zhu~\cite{FOOZ} showed that for every integer $k$ there are trees that have no nonrepetitive vertex-coloring from lists of size $k$. 

Up to this point the only paper we are aware of that narrows the large gap between the trivial lower bound of $\Delta$ colors in a nonrepetitive edge-coloring of a tree of maximum degree $\Delta$ and the $4\Delta-4$ upper bound from~\cite{AGHR} is by Sudeep and Vishwanathan~\cite{SV}. We will describe their results in the next section. The main result of this paper is to give the first nontrivial improvement of the upper bound from~\cite{AGHR}.

\begin{theorem}\label{thm:main}
If $G$ is a tree of maximum degree $\Delta$, then $\pi'(G)\le 3\Delta-2$.
\end{theorem}

We will give a proof of this theorem in Section~\ref{sec:main} using a coloring method we describe in Section~\ref{sec:derived} . We discuss some possible ways for further improvements in Section~\ref{sec:improve}. 


\section{Trees of small height}\label{sec:small}

A $k$-ary tree is a tree with a designated root and the property that every vertex that is not a leaf has exactly $k$ children. The $k$-ary tree in which the distance from the root to every leaf is $h$ is denoted by $T_{k,h}$. For convenience we will assume that the vertices in $T_{k,h}$ are labeled as suggested in Figures~\ref{fig:2,2} and~\ref{fig:2,3} with the root labeled 1, its children labeled $2,\dots,k+1$, their children $k+2,\dots k^2+k+1$ and so on. This allows us to write $u<v$ if $u$ is to the left or above $v$, and also gives the vertices at each level (distance from the root) a natural left to right order.

To obtain bounds on the Thue chromatic index of general trees $G$ of maximum degree $\Delta\ge 2$ it suffices to study $k$-ary trees for $k=\Delta-1$, since $G$ is a subgraph of $T_{k,h}$ for sufficiently large $h$.
Of course the Thue sequence shows that for $h>4$ we have $\pi'(T_{1,h})=\pi'(P_h)=3$, and it is similarly obvious that $\pi'(T_{k,1})=\pi'(K_{1,k})=k$. 
It is easy to see that the next smallest tree $T_{2,2}$ already requires 4 colors, and Figure~\ref{fig:2,2} shows the only two such 4-colorings up to isomorphism. 

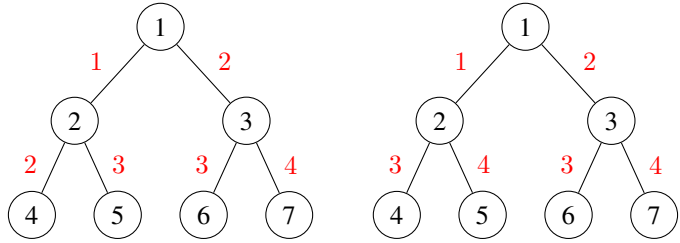
\begin{figure}[htbp]
 \begin{center}
\begin{tikzpicture}[every tree node/.style={draw,circle},
   level distance=1.25cm,sibling distance=.5cm,
   edge from parent path={(\tikzparentnode) -- (\tikzchildnode)}]
\Tree
[.1
    \edge node[auto=right, red,pos=.6] {$1$};
    [.2
       \edge node[auto=right,red,pos=.8] {$2$};
       [.4 ]
       \edge node[auto=left,red,pos=.8] {$3$};
       [.5 ]
        ]
    \edge node[auto=left,red,pos=.6] {$2$};
    [.3
        \edge node[auto=right,red,pos=.8] {$3$};
        [.6 ]
        \edge node[auto=left,red,pos=.8] {$4$};
        [.7 ]
        ]
]
\end{tikzpicture}
\qquad
\begin{tikzpicture}[every tree node/.style={draw,circle},
   level distance=1.25cm,sibling distance=.5cm,
   edge from parent path={(\tikzparentnode) -- (\tikzchildnode)}]
\Tree
[.1
    \edge node[auto=right, red,pos=.6] {$1$};
    [.2
       \edge node[auto=right,red,pos=.8] {$3$};
       [.4 ]
       \edge node[auto=left,red,pos=.8] {$4$};
       [.5 ]
        ]
    \edge node[auto=left,red,pos=.6] {$2$};
    [.3
        \edge node[auto=right,red,pos=.8] {$3$};
        [.6 ]
        \edge node[auto=left,red,pos=.8] {$4$};
        [.7 ]
        ]
]
\end{tikzpicture}
    \caption{Nonrepetitive 4-edge-colorings of $T_{2,2}$ of type I and II.}
    \label{fig:2,2}
  \end{center}
\end{figure}

The Masters thesis of the second author~\cite{thesis} contains a proof of the fact that the type II coloring of $T_{2,2}$ extends to a unique 4-coloring of $T_{2,3}$ whereas the type I coloring extends to exactly 5 non-isomorphic 4-colorings of $T_{2,3}$, one of which we show in Figure~\ref{fig:2,3}. It is furthermore shown that none of these 6 colorings can be extended to $T_{2,4}$. In fact $\pi'(T_{2,4})=5$ as we can easily extend the coloring from Figure~\ref{fig:2,3} by using color 5 on one of the two new edges at every vertex from $8$ through $15$, and (for example) using colors 1,1,3,4,2,3,2,3 on the other edges in this order.

\begin{figure}[htbp]
\begin{center}
\begin{tikzpicture}[every tree node/.style={draw,circle},
   level distance=1.25cm,sibling distance=.5cm,
   edge from parent path={(\tikzparentnode) -- (\tikzchildnode)}]
\Tree
[.1
    \edge node[auto=right, red,pos=.6] {$1$}; 
    [.2
       \edge node[auto=right,red,pos=.8] {$2$}; 
       [.4 
         \edge node[auto=right,red,pos=.8] {$3$}; 
          [.8 ]
          \edge node[auto=left,red,pos=.8] {$4$}; 
           [.9 ]
          ] 
       \edge node[auto=left,red,pos=.8] {$3$}; 
       [.5 
         \edge node[auto=right,red,pos=.8] {$4$}; 
          [.10 ]
          \edge node[auto=left,red,pos=.8] {$1$}; 
           [.11 ]
          ] 
        ] 
    \edge node[auto=left,red,pos=.6] {$2$}; 
    [.3
        \edge node[auto=right,red,pos=.8] {$3$}; 
        [.6 
          \edge node[auto=right,red,pos=.8] {$4$}; 
          [.12 ]
          \edge node[auto=left,red,pos=.8] {$1$}; 
           [.13 ]
          ] 
        \edge node[auto=left,red,pos=.8] {$4$}; 
        [.7 
         \edge node[auto=right,red,pos=.8] {$1$}; 
          [.14 ]
          \edge node[auto=left,red,pos=.8] {$2$}; 
           [.15 ]
          ] 
        ] 
] 
\end{tikzpicture}
    \caption{Nonrepetitive 4-edge-coloring of $T_{2,3}$.}
    \label{fig:2,3}
  \end{center}
\end{figure}
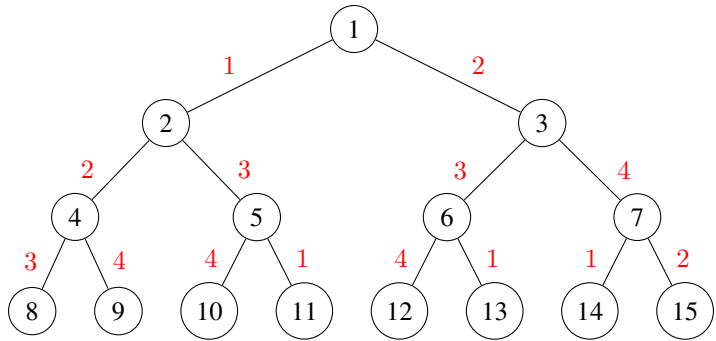

On a more general level, Sudeep and Vishwanathan~\cite{SV} proved that $\pi'(T_{k,2})=\lfloor \frac32 k\rfloor+1$ (compare also Theorem~4 of~\cite{BLMSS}) and $\pi'(T_{k,3})>\frac{\sqrt5 +1}2k>1.618k$. Their lower bounds follow from counting arguments, whereas the construction for $h=2$ consists of giving the edges at the first level colors $0,1,\dots, k-1$ and using all the $\lfloor k/2\rfloor+1$ remaining colors below each vertex at level 1. The remaining $m=\lceil k/2\rceil -1$ edges below the edge of color $i$ are colored with $i+1 \mod k, i+2 \mod k,\dots,i+m \mod k$, in other words cyclically.   

To explain the general upper bound of Alon, Grytczuk, Ha{\l}uszczak and Riordan~\cite{AGHR} we let $T_k$ denote the infinite $k$-ary tree.
It is not difficult to see that $\pi'(T_k)$ is the minimum number of colors needed to color $T_{k,h}$ for every $h\ge 1$. They prove that $\pi'(T_k)\le 4k$ by giving a nonrepetitive edge-coloring of $T_k$ on $4k$ colors as follows: 

Starting with a Thue-sequence $123231\dots$ insert 4 as every third symbol to obtain a nonrepetitive sequence $S=124324314\dots$ that also does not contain a {\em palindrome}, that is a sequence of length at least 2 that reads forwards the same as backwards, such as 121. Now color the edges with a common parent at distance $h-1$ from the root with $k$ different copies $s^{(1)},\dots ,s^{(k)}$ of the symbol $s$ in position $h$ of $S$. For example, the type II coloring in Figure~\ref{fig:2,2} is isomorphic to the first two levels of this coloring of $T_2$ if we replace $1^{(1)},1^{(2)},2^{(1)},2^{(2)}$ by $1,2,3,4$ respectively. It is now easy to verify that this coloring has no repetitively colored paths that are monotone ({\it i.e.} have all vertices at different levels) since $S$ is nonrepetitive, and none with a turning point ({\it i.e.} a vertex whose two neighbors on the path are its children) since $S$ is palindrome-free.   
  
Sudeep and Vishwanathan noted the gap between the bounds $1.618k<\pi'(T_k)\le 4k$, and stated their belief that both can be improved. 
Even for $k=2$ the gap $3.2<\pi'(T_2)\le 8$ is large. Whereas obviously $\pi'(T_2)\ge \pi'(T_{2,4})=5$ is not hard to obtain, the specific question of showing that $\pi'(T_2)<8$ is already raised in~\cite{AGHR} at the end of Section 4.2. Theorem~\ref{thm:main} implies that indeed $\pi'(T_2)\le 7$. On the other hand, improving on the lower bound of 5 (if that is possible) would require different ideas from those in~\cite{SV} because~\cite{thesis} presents a nonrepetitive 5-coloring of $T_{2,10}$ as Example~3.2.6.  


\section{Derived colorings}\label{sec:derived}

In this section, which can also be found in~\cite{thesis}, we present a way to color the edges of $T_k$ that is different from that used by Alon, Grytczuk, Ha{\l}uszczak and Riordan~\cite{AGHR}.
While their idea is in some sense the natural generalization of the type II coloring  in the sense that the coloring precedes by level, our coloring generalizes the type I coloring by moving diagonally. The fact that the type I colorings could be extended in 5 nonisomorphic ways, whereas the extension of the type II coloring was unique encourages this notion.

\begin{definition}
Let $S=s_1,s_2, \dots$ be a sequence. The edge-coloring of a $k$-ary tree $T$ \textbf{derived} from $S$ is obtained as follows:
 The edges incident with the root receive colors $s_1,s_2,\dots ,s_k$ going from left to right in this order. 
If $v$ is any vertex other than the root and if the edge between $v$ and its parent has color $s_i$, then the edges between $v$ and its children receive colors $s_{i+1},s_{i+2},\dots ,s_{i+k}$ again going from left to right in this order. 
\end{definition}

To color the edges of the infinite $k$-ary tree $T_k$ in this fashion we need $S$ to be infinite. To color the edges of $T_{k,h}$ it suffices for the length of $S$ to be at least $kh$ (which is rather small considering that there about $k^h$ edges) as each level will use $k$ entries of $S$ more than the previous level (on the edges incident with the right-most vertex). For example the type I coloring of $T_{2,2}$ is the coloring derived from $S=1,2,3,4$, whereas the coloring of $T_{2,3}$ in Figure~\ref{fig:2,3} is derived from $S=1,2,3,4,1,2$. The next definition will enable us to characterize infinite sequences whose derived coloring is nonrepetitive.

\begin{definition}\label{def:kspecial} 
Let $S=s_1,s_2, \dots $ be a (finite or infinite) sequence. A sequence of indices $i_1,i_2, \dots, i_{2r}$ is called {\bf $k$-bad} for $S$ if there is an $m$ with $1 < m \leq 2r$ such that the following four conditions hold:

\begin{enumerate}[a)]
\item $s_{i_1},s_{i_2}, \dots ,s_{i_{2r}}$ is a repetition 
\item  $i_1 > i_2> \dots > i_m < i_{m+1} < i_{m+2} < \dots < i_{2r}$
\item  $|i_j-i_{j+1}| \leq k$ for all $j$ with $1 \leq j <2r$
\item $i_{m+1}<i_m+k$ if $m<2r$.
\end{enumerate}

$S$ is called \textbf{{\speci}} if it has no $k$-bad sequence of indices.
 \end{definition}

The following proposition says something about the structure of a {\speci} sequence, namely that identical entries must be at least $2k$ apart.

\begin{proposition}\label{prop:distance}
A sequence $S$ has a $k$-bad sequence of length at most four with $m\le 3$ if and only if $s_i=s_j$ for some $i<j< i+2k$.
\end{proposition}

\begin{proof}
For the back direction observe that if $j\le i+k$, then the sequence of indices $j,i$ is $k$-bad with $m=2$.
If $i+k\le j< i+2k$, then the sequence $i+k-1,i,i+k-1,j$ is $k$-bad with $m=2$.

For the forward direction, observe that if $i_1,i_2$ is $k$-bad (necessarily with $m=2$), then we can let $j=i_1$ and $i=i_2$. 
If $i_1,i_2,i_3,i_4$ is $k$-bad with $m=2$ then we let $i=i_2$ and $j=i_4$ and observe that $i<i_3<j\le i_3+k\le i+2k-1$.
So we may assume that $i_1,i_2,i_3,i_4$ is $k$-bad with $m=3$. If $i_2=i_4$, then we let $i=i_3$ and $j=i_1$ and obtain $i<i_2<j\le i_4+k-1=i_2+k-1\le i+2k-1$ as desired.  Otherwise $i_2, i_4$ are distinct numbers $x$ with $i_3<x\le i_3+k$ and we can let $\{i,j\}=\{i_2,i_4\}$.
\end{proof}

We are now ready to prove the following.

\begin{theorem}\label{thm:kspecial}
An infinite sequence $S$ is {\speci} if and only if the edge-coloring of $T_k$ derived from $S$ is nonrepetitive. 
\end{theorem}

\begin{proof} $(\Rightarrow)$ Suppose that a  {\speci}  sequence $S$ creates a repetition on a path $P=v_0,v_1,\dots, v_{2r}$ in $T_k$, that is $R=c(v_0v_1),c(v_1v_2),\dots,c(v_{2r-1}v_{2r})$ satisfies $c(v_iv_{i+1})=c(v_{i+r}v_{i+r+1})$ for $0 \leq i \leq r-1$. Observe that $c(v_jv_{j+1})=s_{i_{j+1}}$ where $0 \leq j \leq 2r-1$, for some $s_{i_{j+1}} \in S$. There are two possibilities; $v_0,v_1,\dots, v_{2r}$ is monotone or it has a single turning point.

\textbf{Case 1:} Suppose $v_0,v_1,\dots, v_{2r}$ is monotone.\\
If $v_0,v_1,v_2 \dots, v_{2r}$ is monotone then we may assume $v_0>v_1>v_2>\dots>v_{2r}$. Since $v_j>v_{j+1}$ we know that $v_j$ is the child of $v_{j+1}$ so we have that $i_j>i_{j+1}$ and $|i_j-i_{j+1}|\leq k$. The subsequence $s_{i_1},s_{i_2},\dots, s_{i_{2r}}$ is a repetition, so that $i_1,\dots,i_{2r}$ is $k$-bad with $m=2r$, a contradiction.

\textbf{Case 2:} Suppose $v_0,v_1,\dots, v_{2r}$ has a turning point $v_m$ for some $m$ with $0<m<2r$.
 By the definition of a turning point $v_{m-1}$ and $v_{m+1}$ are the children of $v_m$, and thus $v_0>v_1>\dots>v_{m-1} >v_m <v_{m+1}< \dots <v_{2r}$. We may also assume without loss of generality that $v_{m-1}<v_{m+1}$. Observe that $v_0,v_1,\dots, v_{m}$ is moving towards the root and $v_m,v_{m+1},\dots, v_{2r}$ is moving away from the root. Let $c(v_jv_{j+1})=s_{i_{j+1}}$. We will show that $i_1>i_2>\dots >i_{m-1}>i_m<i_{m+1}<\dots< i_{2r}$ and that this sequence is $k$-bad for $S$. Since $v_{j-1} > v_j > v_{j+1}$ for $1 \leq j < m$ we know that $v_j$ is the child of $v_{j+1}$ and the parent of $v_{j-1}$ so we have $i_j>i_{j+1}$ and $|i_j-i_{j+1}|\leq k$.  Similarly, since $v_{j-1}<v_j<v_{j+1}$ for $m < j < 2r$ we know that $v_{j}$ is the child of $v_{j-1}$ and the parent of $v_{j+1}$ so $i_j<i_{j+1}$ and $|i_j-i_{j+1}| \leq k$. Finally, since $v_m$ is the parent of $v_{m-1}$ and $v_{m+1}$ so $|i_m-i_{m+1}|<k$ and $i_m < i_{m+1}$ since we assumed $v_{m-1} < v_{m+1}$. The subsequence $s_{i_1},s_{i_2},\dots, s_{i_{2r}}$ is a repetition, leading to the contradiction that $i_1,\dots,i_{2r}$ is $k$-bad.

 $(\Leftarrow)$ We proceed by contrapositive. So suppose $S$ has a $k$-bad sequence ${i_1},{i_2}, \dots ,{i_{2r}}$. We will show that there is a path  on vertices $v_0,v_1,v_2,\dots, v_{2r}$ with $c(v_jv_{j+1})=s_{i_{j+1}}$ where the color pattern $c(v_0v_1),c(v_1v_2)\dots,c(v_{2r-1}v_{2r})$ is a repetition in the derived edge-coloring of $T_k$. The left child of a vertex $v$ is the child with the smallest label, and we will denote this child as $v'$. Observe that if $c(vp(v))=s_\alpha$, then $c(vv')=s_{\alpha+1}$. 

If $m=2r$ then we start at the root and successively go to the left child of the current vertex until we find a vertex $v_{2r}$ such that $c(v_{2r}v_{2r}')=s_{i_{2r}}$ and let $v_{2r-1}=v_{2r}'$. Let $v_{2r-2}$ be the child of $v_{2r-1}$ with $c(v_{2r-1}v_{2r-2})=s_{i_{2r-1}}$ (this exists since  $|i_j-i_{j+1}| \leq k$). We continue in this way until we have found $v_{0}$. Now observe that the color pattern of $v_0,v_1,\dots, v_{2r}$ is  $s_{i_1},s_{i_2}, \dots ,s_{i_{2r}}$ as desired.

 If $m < 2r$ then we start at the root and successively go to the left child of the current vertex until we find a vertex $v_m$ such that $c(v_mv'_m)=s_{i_m}$ and let $v_{m-1}=v_{m}'$. Let $v_{m+1}$ be the child of $v_m$ with $c(v_mv_{m+1})=s_{i_{m+1}}$ (this exists since $i_m<i_{m+1}<i_m+k$).  Now, for $0 \leq p \leq (m-1)$ we successively find a child $v_{p-1}$ of $v_p$ such that $c(v_pv_{p-1})=s_{i_p}$. The existence of $v_{p-1}$ is guaranteed by the fact $|i_p-i_{p-q}|\leq k$ as in the case $m=2r$. For $m+1\leq q \leq 2r$ we successively find a child $v_{q+1}$ of $v_q$ such that $c(v_qv_{q+1})=s_{i_{q-1}}$ which we can do since $|i_q-i_{q+1}|\leq k$. Now observe that the color pattern of $v_0,v_1,\dots, v_{2r}$ is  $s_{i_1},s_{i_2}, \dots ,s_{i_{2r}}$ as desired.
\end{proof}

\begin{remark}\label{rem:finite}
Observe that the proof of the forward direction also works for the finite case $T_{k,h}$, a fact we will use in Section~\ref{sec:improve}. However, the back direction need not hold in this case: We already mentioned that the coloring derived from $S=1,2,3,4,1,2$ in Figure~\ref{fig:2,3} is nonrepetitive (see also $k=2$ in Proposition~\ref{2k}), but this sequence $S$ is not 2-special, because the index-sequence $3,1,2,3,5,6$ is $2$-bad. 
\end{remark}

Thus to get a good upper bound on $\pi'(T_k)$ we just need an infinite {\speci} sequence with few symbols. 
As every $2k$ consecutive elements must be distinct, the following simple idea turns out to be useful: from a sequence $S$ on $q$ symbols we can form a sequence $S^{(w)}$ on $qw$ symbols by replacing each symbol $t$ in $S$ by a block $T=t^{(0)},t^{(1)},\dots t^{(w-1)}$ of $w$ symbols. In~\cite{thesis} it is shown that if $S$ is nonrepetitive and palindrome-free then $S^{(k)}$ is \speci. This gives a new proof of the result from~\cite{AGHR} that $\pi'(T_k)\le 4k$. In the next section we will improve on that.


\section{Main result}\label{sec:main}

We begin with the simple observation, that if $S$ is a sequence then $S^{(k+1)}=S^+$ has the property that if $i,j$ are indices with $s^+_i=x^{(u)}$ and $s^+_j=y^{(v)}$ then $i<j\le i+k$ implies that either $x=y$ and $u<v$, or $s^+_i$ and $s^+_j$ are in consecutive blocks $XY$ of $S^+$ and $u>v$. In other words we can tell whether we are moving left or right through the sequence just by looking at the superscripts (as long as consecutive symbols in $S$ are distinct.) As a starting point we immediately get the following result.

\begin{corollary}\label{cor:3k+3}
For all $k\ge 1$, $\pi'(T_k)\le 3k+3$.
\end{corollary}

\begin{proof}
It is enough to show that $S^+$ on $3(k+1)$ is {\speci} whenever $S$ is an infinite Thue sequence on 3 symbols. Suppose there is a $k$-bad sequence of indices $i_1,\dots,i_{2r}$. Since every sequence of $2(k+1)$ consecutive symbols in $S^+$ is distinct we get that $r>1$ by Proposition~\ref{prop:distance}. If $m<2r$, then we can find an index $j$ such that $i_j>i_{j+1}$ and $i_{r+j}<i_{r+j+1}$ with $s_{i_j}=s_{i_{r+j}}=x^{(u)}$ and $s_{i_{j+1}}=s_{i_{r+j+1}}=y^{(v)}$. Indeed, if $2<m\le r$ we let $j=1$, and otherwise we let $j=m-r$. In this case $x=y$ and $u\le v$ would violate  $i_j>i_{j+1}\ge i_j-k$, whereas $u\ge v$ would violate $i_{r+j}<i_{r+j+1}\le i_{r+j}+k$. Similarly if $x\neq y$, then $u\ge v$ would violate  $i_j>i_{j+1}\ge i_j-k$, whereas $u\le v$ would violate $i_{r+j}<i_{r+j+1}\le i_{r+j}+k$.  

It remains to observe that in the case when $m=2r$ the sequence $s_{i_1},s_{i_2}, \dots ,s_{i_{2r}}$ in $S^+$ yields a repetition in $S$ by erasing the superscripts and merging identical consecutive terms where necessary.  
\end{proof}

This bound can be improved to $3k+2$ by removing all symbols of the form $a^{(0)}$ from $S^+$ for one of the symbols $a$ from $S$ and showing that the resulting sequence is still \speci. However, we can do a bit better. In fact, Theorem~\ref{thm:main} follows directly from our main result in this section.

\begin{theorem}\label{thm:3k+1}
There are arbitrarily long {\speci} sequences on $3k+1$ symbols.
\end{theorem}

One difficulty is that removing two symbols from $S^+$ can easily result in the sequence not being {\speci} anymore. To make the proof work we need to start with a Thue sequence with additional properties. The following result was proved by Thue~\cite{T2} and reformulated by Berstel~\cite{B2,B1} using modern conventions.

\begin{theorem}\label{thm:Thue+}
There are arbitrarily long nonrepetitive sequences with symbols $a, b, c$ that do not contain $aba$ or $bab$.
\end{theorem}

To give an idea of how such a sequence can be found, observe that it must be built out of blocks of the form $ca, cb, cab,$ and $cba$ which we denote by $x, y, z, u$, respectively. (In fact, Thue primarily studied two-way infinite sequences, but for our purposes we may simply assume our sequence starts with $c$.) 
We first build a sufficiently long sequence on the 5 symbols $A, B, C, D, E$ by starting with the sequence "B" and then in each step simultaneously replacing each letter as follows:

\begin{center}
\begin{tabular}{l|l|l|l|l|l}
Replace & A & B & C & D & E\\
\hline
by  & BDAEAC & BDC & BDAE & BEAC & BEAE \\
\end{tabular}
\end{center}

In the resulting sequence we then let $A=zuyxu$, $B=zu$, $C=zuy$, $D=zxu$, $E=zxy$. Lastly we replace $x$, $y$, $z$ and $u$ as aforementioned. 
For example, from $B$ we obtain $BDC$, and then 
after a second step $BDCBEACBDAE$. This translates to the intermediate sequence 

\noindent $zuzxuzuyzuzxyzuyxuzuyzuzxuzuyxuzxy$, which gives us the desired sequence 

\noindent $cabcbacabcacbacabcbacbcabcbacabcacbcabcbacbcacbacabcbacbcabcbacabcacbacabcbacbcacbacabcacb$.

It is worth pointing out that Thue's work goes deeper in that he essentially characterizes all two-way infinite sequences that meet the conditions from Theorem~\ref{thm:Thue+} as well as several other related sequences. We also want to mention that the $A,B,C$ in the following proof have nothing to do with the $A,B,C$ in the previous paragraph, but we wanted  to maintain the notation used in~\cite{B2,B1}.

\begin{proof}[of Theorem~\ref{thm:3k+1}]
Start with an infinite sequence $S$ in the form of Theorem~\ref{thm:Thue+} and replace each occurrence of $c$ by a block $C$ of $k+1$ consecutive symbols $c^{(0)},c^{(1)},\dots,c^{(k)}$, whereas we replace each occurrence of $a$ or $b$ by shorter blocks $A=a^{(1)},\dots,a^{(k)}$ and $B=b^{(1)},\dots,b^{(k)}$ respectively. We claim that the resulting sequence $S'$ on $3k+1$ symbols is \speci. So suppose there is a $k$-bad sequence of indices $i_1,\dots,i_{2r}$. As before when $m=2r$ the sequence $s_{i_1},s_{i_2}, \dots ,s_{i_{2r}}$ in $S'$ yields a repetition in $S$ by erasing the superscripts and merging identical consecutive terms where necessary, as we can not "jump" over any of the blocks $A, B$ or $C$ in $S'$. So we may assume that $1<m<2r$, and since every $2k$ consecutive elements are distinct Proposition~\ref{prop:distance} implies that $r>2$. 

{\bf Claim:} If there is an index $j$ with $0<j< r$ such that $i_j>i_{j+1}$ and $i_{r+j}<i_{r+j+1}$, then $s_{i_j}=s_{i_{r+j}}=x^{(u)}$ and $s_{i_{j+1}}=s_{i_{r+j+1}}=y^{(u)}$ for $1\le u\le k$ and $\{x,y\}=\{a,b\}$. Consequently, $i_j-i_{j+1}=k=i_{r+j+1}-i_{r+j}$.

Indeed, $s_{i_j}=s_{i_{r+j}}=x^{(u)}$ and $s_{i_{j+1}}=s_{i_{r+j+1}}=y^{(v)}$ for some $u,v,x,y$. If $x=y$, then $u\le v$ would violate  $i_j>i_{j+1}\ge i_j-k$, whereas $u\ge v$ would violate $i_{r+j}<i_{r+j+1}\le i_{r+j}+k$. Thus $x\neq y$. Now $u> v$ would violate  $i_j>i_{j+1}\ge i_j-k$, whereas $u< v$ would violate $i_{r+j}<i_{r+j+1}+k$. So we may assume that $u=v$. If $x=c$, then this would violate $i_{j}>i_{j+1}\ge i_j+k$ (as the presence of $c^{(0)}$ means that the distance is $k+1$). Similarly if $y=c$, then this violates $i_{r+j}<i_{r+j+1}\le i_{r+j}+k$. Hence we must have $\{x,y\}=\{a,b\}$ finishing the proof of the claim.

If $r<m<2r$, then we can apply the claim with $j=m-r$ and obtain consequently that $i_{m+1}-i_m=k$, in direct contradiction to condition d) from Definition~\ref{def:kspecial}.

So we suppose that $2 \le m\le r$. In this case we will let $j=m-1$ in our claim and we may assume due to the symmetry of $S$ in $a,b$ that $x=a$ and $y=b$. Thus for some $u$ with $1\le u\le k$ we get $s_{i_{m-1}}=a^{(u)}=s_{i_{m+r-1}}$ and $s_{i_m}=b^{(u)}=s_{i_{m+r}}$. If $m>2$, then we may apply the claim again with $j=m-2$ to obtain that $s_{i_{m-2}}=b^{(u)}=s_{i_{m+r-2}}$. However, the fact that $i_{m-2}>i_{m-1}>i_m$ correspond to symbols $b^{(u)},a^{(u)},b^{(u)}$ means that $S'$ must have consecutive blocks $BAB$, yielding a contradiction to the fact that in $S$ we had no consecutive symbols $bab$. 

So we may assume that $m=2$. 
Since $r>2$ and $s_{i_2}=b^{(u)}$ and $i_2<\dots<i_r$ we have that for $3\le j\le r$ either all $s_{i_j}$ are of the form $b^{(u_j)}$ or there is a smallest index $j$ such that $s_{i_j}=x^{(u_j)}$ for some $x\neq b$. In the first case it follows that there must be consecutive blocks $BAB$ (yielding a contradiction) such that $i_1$ and $i_{r+1}$ are in the $A$ block, $i_2,\dots i_r$ are in the first $B$-block and $i_{r+2},\dots,i_{2r}$ are in the second. In the second case it follows that since there must be blocks $BA$ with $i_1$ in $A$ and $i_2$ in $B$, that $i_j$ must be in the $A$ block again, that is $s_{i_j}=a^{(u_j)}$. However, since $i_{r+1}<\dots <i_{r+j}$ it follows that there must be consecutive blocks $ABA$ in $S'$ (our final contradiction), such that $i_{r+1}$ is in the first $A$ block, $i_{r+j}$ in the second and $i_{r+2}, \dots, i_{r+j-1}$ are in the $B$ block.  
\end{proof}


\section{{\speci} sequences on at most $3k$ symbols}\label{sec:improve}

One possible way to improve on Theorem~\ref{thm:main} is to study {\speci} sequences on at most $3k$ symbols. 
The sequence $S_{n,c}=1,2,\dots,n,1,2,\dots  c$ for $n>c\ge 0$ turns out to be a key example in this situation.

Recall that by Proposition~\ref{prop:distance} the entries in a block of length $2k$ of a {\speci} sequence must all be distinct. 
Thus, if we let $f_k(n)$ denote the maximum length of a {\speci} sequence $S$ on $n$ symbols, then this observation immediately implies that $f_k(n)=n$ when $n<2k$ and up to isomorphism the only sequence achieving this value is $S_{n,0}$. When $n\ge 2k$ we can furthermore assume without loss of generality that if $S$ is nonrepetitive on $n$ symbols, then $S_i=i$ for $1\le i\le 2k$ (just like $S_{n,c}$.) 

If $n=2k$ then it follows from Proposition~\ref{prop:distance} that a sequence achieving $f_k(2k)$ must be of the form $S_{2k,c}$. It is easy to check $S_{2k,1}$ is in fact \speci, whereas $S_{2k,2}$ contains the $k$-bad index sequence $k+1,1,2,k+1,2k+1,2k+2$, which yields the repetition $k+1,1,2,k+1,1,2$.  Thus $f_k(2k)=2k+1$ with $S_{2k,1}$ being the unique sequence achieving this value. 
This $k$-bad index sequence also explains why we could not have consecutive blocks $ABA$ or $BAB$ in our construction for Theorem~\ref{thm:3k+1} .
For the remaining range we get

\begin{proposition}\label{prop:2k+}
~
\begin{enumerate}[a)]
\item If $n\ge 2k$, then $S_{n,n-k}$ has a $k$-bad sequence only when $n=2k$ and such a sequence must have $2=m<r$.
\item If $n\ge 2k+1$, then $f_k(n)\ge 2n-k$.
\end{enumerate}
\end{proposition}

\begin{proof}
It suffices to prove the first statement, as it immediately implies the second. So suppose $n\ge 2k$ and $I=i_1,\dots,i_{2r}$ is a $k$-bad sequence of indices for some $m$. 
If $m=2r$, then $I$ is decreasing and so the fact that $s_{i_j}=s_{i_{j+r}}$ for all $1\le j\le r$ implies that $i_1>\dots >i_r\ge n+1$ and $n-k\ge i_{r+1}>\dots>i_{2r}$, yielding the contradiction $i_r-i_{r+1}>k$. So we may assume that $m<2r$.
%

If $m>r$, then let $m'=m-r$. Since $s_{i_m}=s_{i_{m'}}$ and $i_{m'}>i_m$, it follows that $i_{m}=i_{m'}-n\in\{1,\dots,n-k\}$. Since $i_{m'}\ge n,$  $i_m\le n-k$ and for all $j$ we have $|i_j-i_{j+1}|\le k$ it follows that there must be some $j$ with $m'<j<m$ such that $i_j\in\{n-k+1,\dots, n\}$.  Since $I$ yields a repetition with $i_1>\dots >i_m$, but the symbol $s_{i_j}=i_j$ is unique in $S_{n,n-k}$ we conclude that $i_j=i_{j+r}$. 
It follows that $j=m'+1$, since otherwise $i_{m'}>i_{j-1}>i_{j}$ and $i_m<i_{j+r-1}<i_{j+r}$ would contradict $s_{i_{j-1}}=s_{i_{j+r-1}}$ as the sets $\{s_{i_j+1},s_{i_j+2},\dots, s_{i_{m'}-1}\}$ and  $\{s_{i_m+1}, s_{i_m+2},\dots s_{i_j-1}\}$ are disjoint. Now $j=m'+1$ implies that
$i_{m'}-k=i_{j-1}-k\le i_j=i_{j+r}=i_{m+1}\le i_m+k-1$, and since $i_{m'}=i_m+n$ we get $n\le 2k-1$, a contradiction.

If $m\le r$, then let $m'=m+r$. It follows again that $i_{m'}=i_m+n$, and that there must be some $j$ such that $i_j=i_{j+r}\in\{n-k+1,\dots, n\}$ and $j<m<j+r$. Thus $m'>j+r$ this time. It follows that $j=m-1$, since otherwise $i_{j}>i_{j+1}>i_{m}$ and $i_{j+r}<i_{j+r+1}<i_{m'}$ would contradict $s_{i_{j+1}}=s_{i_{j+r+1}}$ as the sets  $\{s_{i_m+1}, s_{i_m+2},\dots s_{i_j-1}\}$ and  $\{s_{i_j+1},s_{i_j+2},\dots, s_{i_{m'}-1}\}$ are still disjoint. Now $j=m-1$ implies that $i_{m}+k=i_{j+1}+k\ge i_j=i_{j+r}=i_{m'-1}\ge i_{m'}-k$, and since $i_{m'}=i_m+n$ we get $n\le 2k$, a contradiction unless $n=2k$. In this case also $i_m+k=i_j=i_{j+r}=i_{m'}-k=x$ for some $k+1\le x\le n=2k$.

If we have $m>2$ then $j-1=m-2\ge 1$ and we consider $i_{j-1}$. Since $i_{j+r-1}<i_{j+r}$ and $k+1=n-k+1\le s_{i_j}\le n=2k$ implies that $s_{i_{j+r-1}}\in\{x-k,x-k+1,\dots,x-1\}$.  Similarly $i_{j-1}>i_j$ implies that $s_{i_{j-1}}\in\{x+1,x+2,\dots,n\}\cup\{1,2,\dots,k-(n-x)=x-k\}$. Since $s_{i_{j+r-1}}=s_{i_{j-1}}$ it now follows that this value must be $x-k=i_m$. Hence $i_{j+r-1}=i_m$ and thus $m=j+r-1=(m-1)+r-1$. This implies the contradiction $2=r\ge m>2$. Hence $m=2$ and the fact that $r>2$ follows from Proposition~\ref{prop:distance} and the fact that the distance between identical labels is $2k$.
\end{proof}

We believe that for in Proposition~\ref{prop:2k+} b) equality holds  when $2k<n<3k$. An exhaustive search by computer shows that this is the case when $2k<n<3k$ with $n\le 16$. Moreover $S_{2k+1,k+1}$ turns out to be the unique sequence achieving $f_k(2k+1)=3k+2$, whereas for $2k+2\le n<3k$ a typical sequence achieving $f_k(n)$ is obtained by permuting the last $n-k$ entries of $S_{n,n-k}$.

\begin{proposition}\label{2k}
The coloring of $T_{k,3}$ derived from $S_{2k,k}$ is nonrepetitive.
\end{proposition}

\begin{proof}
If the coloring of $T_{k,3}$ derived from $S_{2k,k}$ contains a repetition of length $2r$, then as in the proof of Theorem~\ref{thm:kspecial} it follows that there must be a $k$-bad sequence of $2r$ indices. From Proposition~\ref{prop:2k+} a) it now follows that $r> m=2$. Since a longest path in $T_{k,3}$ has 6 edges we must have $r=3$. However, any repetition of length 6 would have to connect two leaves and turn around at the root, and as such would have $m=3$, a contradiction.   
\end{proof}

Combining everything we know so far we get

\begin{corollary}\label{cor:small}
If $h\ge 3$, then $\pi'(T_{k,h})\le \lceil\frac{h+1}2k\rceil$.
\end{corollary}

\begin{proof}
If $h=3$, then the result follows from Proposition~\ref{2k}. For $h>3$ we can apply Proposition~\ref{prop:2k+} b) with $n=\lceil\frac{h+1}2k\rceil$. Since $2n-k\ge hk$ it now follows from Remark~\ref{rem:finite} that the coloring of $T_{k,h}$ derived from $S_{n,n-k}$ is nonrepetitive. 
\end{proof}

The bound in Corollary~\ref{cor:small} is better than that derived from Theorem~\ref{thm:3k+1} when $h\le 5$ and we obtain the following table of values for $\pi'(T_{h,k})$, where the presence of two values denotes a lower and an upper bound. The values marked by an asterisk were confirmed by computer search. The programs used are based on those found in~\cite{thesis} and the Python code is available at  http://public.csusm.edu/akundgen/Python/Nonrepetitive.py 

\begin{table}[h]
\begin{tabular}{|c|c|c|c|c|c|c|c|c|}
\hline
$k\backslash h$   	& 1  & 2  & 3  & 4 & 5  & 6-10 & $h\geq 11$  \\ \hline
1       			& 1  & 2  & 2  & 3 & 3  & 3      & 3 \\ \hline
2       			& 2  & 4  & 4  & 5 & $5^*$  & $5^*$      & 5,7 \\ \hline
3      				& 3  & 5  & $6^*$   &  $6^*$  & 6,9   & 6,10 & 6,10      \\ \hline
4      				& 4  & 7  & $7^*$   & 7,10 & 7,12 & 7,13 & 7,13      \\ \hline
5      				& 5  & 8  & 9,10 & 9,13 & 9,15 & 9,16 & 9,16      \\ \hline
$\vdots$ & $\vdots$ & $\vdots$                 & $\vdots$  & $\vdots$ & $\vdots$  & $\vdots$ & $\vdots$ \\ \hline
$k$     & $k$     & $\lfloor 1.5k \rfloor$+1 & $1.61k, 2k$ & $1.61k,\lceil2.5k\rceil$ & $1.61k,3k$ & $1.61k,3k+1$ & $1.61k,3k+1$ \\  \hline
\end{tabular}
\end{table}

It is worth noting that even though it may be possible to use derived colorings to improve individual columns of this table by a more careful argument (as we did in Proposition~\ref{2k}), this seems unlikely to work for $\pi'(T_k)$ in general. Theorem~\ref{thm:kspecial} implies that the infinite sequence from which we derive the coloring must be {\speci}, and while we were able to provide such a sequence on $3k+1$ symbols, it seems unlikely that there are such sequences on $3k$ symbols. An exhaustive search shows that for $k\le 5$ the maximum length of a $k$-special sequence on $n=3k$ symbols is $5k+3$, which is only 3 more than the length of  $S_{n,n-k}$. The $k!$ examples achieving this value are all of the strange form $[1,2k],1,[2k+1,3k],x_1,[k+2,2k],1,x_2,x_3,\dots,x_k,x_1,2k+1$ where $\{x_1,\dots,x_k\}=\{2,\dots,k+1\}$ and $[a,b]$ denotes $a,a+1,a+2\dots,b$. In other words they are $S_{3k,2k+1}$ with the last $2k+1$ entries permuted and with 1 and $x_1$ inserted after positions $2k$ and $3k$.

A more promising next step would be to try to improve the lower bounds for $\pi'(T_{k,h})$ for $h=3,4,5$.


\nocite{*} 

\bibliography{repfree-tree-coloring-final}
\label{sec:biblio}

\end{document}